\documentclass[a4paper,11pt]{amsart}

\usepackage{amsfonts,amsmath,mathrsfs,amssymb,amsthm,amscd,latexsym,amstext,amsxtra}

\usepackage[usenames]{color}
\usepackage[all]{xy}
\usepackage{enumerate}
\xyoption{all}
\usepackage{srcltx} 

\usepackage{amsthm, amsmath, amssymb,latexsym}
\usepackage{amsfonts,epsfig,amscd}
\usepackage[]{fontenc}
\usepackage{xy}
\usepackage{enumerate}
\usepackage[]{fontenc}
\usepackage{xy}
\newtheorem{theorem}{Theorem}[section]
\newtheorem{lemma}[theorem]{Lemma}

\newtheorem{proposition}[theorem]{Proposition}
\newtheorem{remark}[theorem]{Remark}
\newtheorem{definition}[theorem]{Definition}

\newcommand{\nc}{\newcommand}
\nc{\cH}{{\mathcal H}}
\nc{\cA}{{\mathcal A}}
\nc{\cG}{{\mathcal G}}
\nc{\cC}{{\mathcal C}}
\nc{\cO}{{\mathcal O}}
\nc{\cI}{{\mathcal I}}
\nc{\cB}{{\mathcal B}}
\nc{\cY}{{\mathcal Y}}
\nc{\cK}{{\mathcal K}}
\nc{\cX}{{\mathcal X}}
\nc{\cS}{{\mathcal S}}
\nc{\cE}{{\mathcal E}}
\nc{\cF}{{\mathcal F}}
\nc{\cZ}{{\mathcal Z}}
\nc{\cQ}{{\mathcal Q}}
\nc{\cN}{{\mathcal N}}
\nc{\cP}{{\mathcal P}}
\nc{\cL}{{\mathcal L}}
\nc{\cM}{{\mathcal M}}
\nc{\cT}{{\mathcal T}}
\nc{\cW}{{\mathcal W}}
\nc{\cU}{{\mathcal U}}
\nc{\cJ}{{\mathcal J}}
\nc{\cV}{{\mathcal V}}
\nc{\bH}{{\mathbb H}}
\nc{\bA}{{\mathbb A}}
\nc{\bG}{{\mathbb G}}
\nc{\bC}{{\mathbb C}}
\nc{\bO}{{\mathbb O}}
\nc{\bI}{{\mathbb I}}
\nc{\bB}{{\mathbb B}}
\nc{\bY}{{\mathbb Y}}
\nc{\bK}{{\mathbb K}}
\nc{\bX}{{\mathbb X}}
\nc{\bS}{{\mathbb S}}
\nc{\bE}{{\mathbb E}}
\nc{\bF}{{\mathbb F}}
\nc{\bZ}{{\mathbb Z}}
\nc{\bQ}{{\mathbb Q}}
\nc{\bN}{{\mathbb N}}
\nc{\bP}{{\mathbb P}}
\nc{\bL}{{\mathbb L}}
\nc{\bM}{{\mathbb M}}
\nc{\bT}{{\mathbb T}}
\nc{\bW}{{\mathbb W}}
\nc{\bU}{{\mathbb U}}
\nc{\bD}{{\mathbb D}}
\nc{\bJ}{{\mathbb J}}
\nc{\bV}{{\mathbb V}}
\nc{\bbZ}{{\mathbb Z}}
\nc{\bR}{{\mathbb R}}
\nc{\fr}{{\rightarrow}}
\nc{\co}{{\nabla}}

\nc{\cu}{{\overlineline{\nabla}}}
\title {On subfields of the function field of a general surface in $\bP^3$}
\author{Yongnam Lee and Gian Pietro Pirola}
\date{}

\address{Department of Mathematical Sciences, KAIST, 291 Daehak-ro, Yuseong-gu, Daejon 305-701, Korea}

\email{ynlee@kaist.ac.kr}

\address{Dipartimento di Matematica, Universit\`a di Pavia via Ferrata 1, 27100 Pavia, Italia}

\email{gianpietro.pirola@unipv.it}


\begin{document}

\begin{abstract}
In this paper we study birational immersions from a very general smooth plane curve to a non-rational surface with $p_g=q=0$ to treat dominant rational maps from a very general surface $X$ of degree$\geq 5$ in $\bP^3$ to smooth projective surfaces $Y$. Based on the classification theory of algebraic surfaces, Hodge theory, and deformation theory, we prove that there is no dominant rational map from $X$ to $Y$ unless $Y$ is rational or $Y$ is birational to $X$.
\end{abstract}
\maketitle



Riemann-Hurwitz Theorem (Chapter XXI in \cite{ACG}) says that if $\phi: C\to C'$ is a non-constant morphism from a very general curve $C$ of genus $g>1$ onto a curve $C'$ then either $\phi$ is birational, or else $C'$ is rational. It is interesting to investigate the same statement for higher dimensional varieties of general type under some assumption of generality in a suitable moduli space. Let $X$ be a smooth complex projective variety of general type. The dominant rational maps of finite degree $X\dashrightarrow Y$ to smooth varieties of general type, up to birational equivalence of $Y$ form a finite set. The proof follows from the approach of Maehara \cite{Ma}, combined with the results of Hacon and McKernan \cite{HM}, of Takayama \cite {Ta}, and of Tsuji \cite{Ts}.

Motivated by this finiteness theorem for dominant rational maps on a variety of general type and by the results obtained in \cite{gp} we study dominant rational maps from a very general complex surface $X$ of degree $d \geq 5$ in $\bP^3$ to smooth projective surfaces $Y$. The main result of this paper is the following.

\begin{theorem} \label{main}
Let $X\subset \bP^3 $ be a very general smooth complex surface of degree $d>4.$
Then there is no dominant rational map $f$ from $X$ to any non rational surface $Y$ unless $f$ is a birational map.
\end{theorem}

We recall  that a very general element of $U$ has the property $P$ if $P$ holds in the complement of a union of countably many
proper subvarieties of $U.$ We get immediately the following completely algebraic version of our theorem.

\begin{theorem}
Let $K$ be the function field of a very general complex surface in $\bP^3$ of degree $d>4.$  Let $\bC\subsetneq K' \subsetneq K$ be a proper subfield of $K.$ Then $K'$ is isomorphic either to   $\bC(x)$, if the transcendence degree of $K'$  is $1$, or to  $\bC(x, y)$ if $K'$ has transcendence degree $2$.
\end{theorem}

If one chooses a special surface $X$ in $\bP^3$ then it might have a dominant rational map to a surface of general type $Y$. Classical Godeaux surfaces $Y$ (minimal surfaces of general type with $p_g(Y):= h^0(Y, K_Y)=0$, $q(Y):= h^1(Y, \cO_Y)=0$, $K^2=1$, and $\pi_1(Y)=\bZ_5$) are obtained by the $\bZ_5$-quotient of $\bZ_5$-invariant quintics \cite{Go}. \\

As far as we know this gives the first examples of fields of transcendence degree $2$ of non-ruled surfaces that do not contain any proper non-rational field. This could have applications to field theory and to absolute Galois theory.\\

To prove our main Theorem, we study birational immersions from a very general smooth plane curve to a non-rational surface with $p_g=q=0.$

\begin{theorem} \label{immersion}
If $C$ is a very general smooth plane curve of degree $d\geq 9$ then there is no birational immersion $\kappa$ from $C$ into any non-rational surface $S$ with $p_g=q=0.$
\end{theorem}

Moreover, we obtain

\begin{theorem} ($=$Theorem~\ref{elliptic})
Let $C$ be a very general smooth plane curve of degree $d\geq 6$. Then there is no birational immersion $\kappa$ from $C$ into any elliptic surface $S$ with $p_g=q=0$ of Kodaira dimension $1$ if ${\rm Pic}(S)$ is torsion free.
\end{theorem}

The method of proof combines the classification theory of algebraic surfaces, Hodge theory and deformation theory.
By Hodge theory (as Section 3.5 in \cite{gp})  one has to consider only dominant rational maps $f: X\dashrightarrow Y$
where $Y$ is simply connected and without  $2-$holomorphic global forms, that is $p_g(Y)=0.$
From the classification theory one obtains that the Kodaira dimension of $Y$, denoted by ${\rm kod}(Y)$, must be $\geq 1.$
The case of surfaces of general type, that is ${\rm kod}(Y)=2,$ was also considered in \cite{gp}, where the problem was solved only for $d\leq 11.$
The new idea here is to consider the full family of smooth plane curves which are hyperplane sections
of surfaces $X\subset \bP^3.$ We show that a very general plane curve, possibly up to some small degree cases,
cannot be birationally immersed in $Y.$
The case of elliptic surfaces, that is when ${\rm kod}(Y)=1,$ is similar, but slightly more difficult.  A careful study of curves on $Y$ as well as an estimate of the moduli of non-rational elliptic surfaces is necessary.

\medskip

The method presented here can be used to obtain similar results for a very general point of families of surfaces of general type that contain families of curves of high dimension. Such are  for instance symmetric products of curves of genus $g\geq 4$ and products of curves (see \cite{BP}). In the paper \cite{LeeP}, we treat a dominant rational map from a product of curves. It is not clear at the moment how to treat the cases of ${\rm kod}(X)= 1$ and ${\rm kod}(X)=0.$ For instance the case where $X$ is a very general quartic surface in $\bP^3$ needs a different approach.

\medskip

In this paper we work over the field of complex numbers.

\subsection*{Acknowledgements}
This work was initiated when the first named author visited University of Pavia supported by INdAM (GNSAGA). He would like to thank University of Pavia for its hospitality during his visit.
The first named author is partially supported by the National Research Foundation of Korea(NRF) funded by the Korea government(MSIP) (No.2013006431) and (No.2013042157).
The second named author is partially supported by INdAM (GNSAGA); PRIN 2012 \emph{``Moduli, strutture geometriche e loro applicazioni''} and FAR 2013 (PV) \emph{``Variet\`a algebriche, calcolo algebrico, grafi orientati e topologici''}. Finally we thank the referees for several useful suggestions and remarks.

\medskip

\section{Family of curves on surfaces}

We will consider birational immersions of curves into surfaces. We begin with the following:
\begin{definition} \label{bim}
Let $C$ be a smooth curve and let $S$ be a smooth projective surface. A morphism $\kappa: C\to S$ will be called a birational immersion if the induced map $C\to\kappa(C)$ is birational.
\end{definition}
Our aim is to show that a very general plane curve $C$ of degree $d\geq 9$ cannot be birationally immersed  into any projective surface  $S$ with $p_g(S)=q(S)=0$ and  Kodaira dimension ${\rm kod}(S)\geq 0.$ For this we  will compare deformations of a plane curve $C$ of degree $d$ and deformations of the map $\kappa.$
\bigskip

We now explain our method starting with some general remarks. Firstly we recall that we can find a countable number of families that contain all the algebraic (smooth) projective varieties.
This  follows, for instance, from the fact that the Hilbert polynomials are countable
and that any Hilbert scheme has a finite number of irreducible components.
By the same kind of argument we can find a countable number of families for all the birational immersions
$\kappa: C\to S,$ where $C$ is a smooth plane curve of degree $d,$  $S$ is any smooth projective surface with $p_g(S)=q(S)=0$ and  Kodaira dimension ${\rm kod}(S)\geq0.$

By contradiction we assume that for a very general smooth plane curve of degree $d,$ $C,$ a birational immersion  $\kappa: C\to S$ exists.
A Baire's category argument shows that there should exist a family of deformations of the birational immersion $\kappa$  that dominates the moduli space of the plane curves of degree $d$.
Therefore, there should be two  smooth families  $\pi: \cC \to W,$   $p: \cS\to W$ and a family map $K: \cC\to \cS,$ where $W$ is smooth variety,   $u\in W$ such that $\pi^{-1}(u)\cong C,$  $p^{-1}(u)\cong S$ and  $K_u\cong \kappa.$
Let $M(d)$  be the moduli space of smooth plane curves of degree $d$. Under our assumption the forgetful map:
$W\to M(d)$ $$x\to moduli[\pi^{-1}(x)]$$ must contain an open set of $M(d).$
Then we can compare the dimension of  $W$  with the dimension $n$ of the vector space of the first order deformations of the map $\kappa.$
We must have:
$$n\geq \dim M(d)= \frac{(d+1)(d+2)}{2}-9.$$
   Next we give a bound $n\leq n'+m'$ where $n'$ is the dimension of the vector space of the  first order deformation of the map $\kappa:C\to S$ where the target surface $S$ is fixed, and $m'=h^1(T_S)$ is the space of the first order deformation of the complex structure of $S.$ Finally if we can prove that
$$m'+n'<\frac{(d+1)(d+2)}{2}-9$$ then we will achieve our contradiction.

In this section we will show this for surfaces $S$ of general type and for Enriques surfaces. In the next section we will adapt the argument to the case of surfaces of Kodaira dimension ${\rm kod}(S)=1$.
\bigskip

We recall a basic result on deformations of curves on a fixed surface (see for instance \cite{ac}). Also we refer to \cite{cata} for a basic result on moduli of algebraic surfaces.
Let $C$ be a smooth projective curve of genus $g(C)\ge 2$ and let $S$ be a smooth projective surface. Let $\kappa:C\to S$ a birational immersion.
Let $T_C$ and $T_S$ be the holomorphic tangent bundles.
 The differential $d\kappa :T_C\to \kappa^\ast T_S$ of a birational immersion $\kappa$ is a sheaf inclusion and hence induces an exact sequence on $C$:
\[ 0\to T_C\to \kappa^\ast T_S\to N\to 0.\]
The first order deformations of $\kappa$ are classified by the global sections $H^0(C,N)$ of the normal sheaf $N.$
Let $U$ be a Kuranishi space of deformations of $\kappa$. When $\kappa$ is not rigid, that
is $\dim(U)>0$, we have a good bound on the dimension of $U$: Replace $\kappa: C\to S$ by a very general point of $U$ if it is necessary.
Let $N_{tors}\subset N$ be the torsion sheaf of $N$ and set $N'=N/N_{tors}.$ Since $\kappa$ is a very general birational immersion,
a simple\footnote{Ciro Ciliberto explains in this way: in an actual deformation of a subvariety its general point $p$ must move, therefore the section of the normal corresponding to the first order deformation cannot vanish at $p$.}, but basic result  (see \cite{ac} Corollary (6.11), and \cite{ACG} Chapter XXI ) gives that $\dim U\leq h^0(C,N').$
When $\deg \kappa^\ast(K_S)=m\geq 0$ we have  $\deg N'\leq 2g-2-m$ where $g=g(C)$. Therefore Clifford theorem gives
\[h^0(N')\leq g-\frac{1}{2}m.\]
We have:

\begin{proposition} \label{cliff} Assume that $K_S$ is a nef divisor. Then \\
  $\dim U\leq \max(0,g-\frac{\deg \kappa^\ast(K_S)}{2}).$
 Moreover if $\dim U=g-\frac{\deg \kappa^\ast(K_S)}{2},$ then $N=N'$ and additionally either $\kappa^\ast(K_S)=\cO_C$ or  $C$ is  hyperelliptic.
\end{proposition}
\begin{proof} If $\dim(U)=0$ then there is nothing to be proved. So we may assume $\dim(U)>0$ then $\dim U\leq h^0(N')\leq  h^0(K_C\otimes K_{S}^{-1}).$ Since $K_S$ is nef the statement follows from  Clifford theorem.\end{proof}

We give an immediate application of the above result:

\begin{proposition} \label{gen} If $C$ is a very general smooth plane curve of degree $d\geq 9$ then there is no birational immersion
$\kappa$ from $C$ into any surface $S$ of general type with $p_g=q=0.$  \end{proposition}

\begin{proof}
 We can assume that $S$ is minimal. The proof is obtained by contradiction: suppose that a very general plane curve $C$ can be birationally immersed in $S$
with $p_g=q=0.$
Plane curves of degree $d$ depend on
 \[ \frac {(d+1)(d+2)}{2}-9\]
 dimensional moduli.
By (\cite{gp}, Corollary 2.5.3) surfaces of general type with $p_g=q=0$ depend on $M\leq 19$
 parameters. It follows that on $S$ we can find at least
 \[ \frac {(d+1)(d+2)}{2}-28\] dimensional family of birational immersion of plane curves of degree $d$. Since this number is positive if $d\ge 7$, and
since $C$ is not hyperelliptic we have then by Proposition~\ref{cliff}
  \[ \frac {(d+1)(d+2)}{2}-28<g-\frac{\deg \kappa^\ast(K_S)}{2}.\]
  if $d\ge 7$. Therefore
  \[\frac {(d+1)(d+2)}{2}-\frac {(d-1)(d-2)}{2}<28-\frac{\deg \kappa^\ast(K_S)}{2}:\]
 \[ 3d< 28-\frac{\deg \kappa^\ast(K_S)}{2}\le 27\]
\noindent since $K_S$ is nef and big.
 Therefore, if $d\geq 9$ then there is no birational immersion $\kappa: C\to S.$ \end{proof}

By the same argument in Proposition~\ref{gen}, one can treat also the case when $S$ is any Enriques surface. We note that the dimension of moduli space of Enriques surfaces is 10.

\begin{proposition} \label{Enriques} If $C$ is a very general smooth plane curve of degree $d\geq 7$ then there is no birational immersion
$\kappa$ from $C$ into any Enriques surface $S$. \end{proposition}


\section{Elliptic surfaces with $p_g=q=0$ of Kodaira dimension 1}
In this section, we study a birational immersion from a very general smooth plane curve to an elliptic smooth projective surface with $p_g=q=0$ of Kodaira dimension $1$. We recall that
 a  surface $S$ (see \cite{bpv})
is an elliptic surface if it admits an elliptic fibration, i.e. there is a smooth projective curve $B$ and
a surjective map \[\pi: S\to B\] such that the general fiber $F=\pi^{-1}(b)$ for $b\in B$ is an elliptic
curve.

We will assume moreover that
\begin{enumerate}
\item $p_g(S)= \dim H^0(S,K_S)=0$.
\item $q(S)= \dim H^1(S, \cO_S)=0$.
\item $S$ is minimal.
\end{enumerate}

Easy consequence of the above conditions, we have $\chi(\cO_S)=1,$ $K^2=0$, $c_2(S)=12$, and since the irregualarity is zero $B=\bP^1$. Let $\pi: S\to\bP^1$ be the elliptic fibration,  and let $N$ be the number of multiple fibers of $\pi$,
and let $k_i$ for $i=1, \ldots, N$ be  their multiplicities.

We first prove the following:
\begin{lemma} \label{elliptic-1}
Let $\pi:S\to \bP^1$ be an irrational minimal elliptic surface with $p_g=q=0$
and let $N$ be the number of multiple fibers  of $\pi$. Suppose that there is a birational immersion $\kappa$ from a very general smooth plane curve $C$ of degree $d\ge 4$ to $S$. Then we have $N\leq d-2.$
\end{lemma}
\begin{proof} We consider the map $f=\pi\circ \kappa: C\to \bP^1.$ Let $\alpha$ be the degree of $f$.
Let $g$ be the genus of $C$. Then from the Hurwitz formula we have
\[ 2g-2\ge -2\alpha+ (\sum^N_{i=1} \alpha(1-1/k_i))+r\geq -2\alpha+N\frac{\alpha}{2}+r \]
where $r$ is the number of branch points which are not in any multiple fiber. Note that the ramification index corresponding to the points in a multiple fiber
is $\geq \frac{\alpha}{2}.$
Clearly $\alpha\geq d-1$, the gonality of the smooth plane curve $C$ (see for instance Chapter I in \cite{ACGH}).

We note that if we fix the number $N$ of the multiple fibers then there are countably many deformation types of irrational minimal elliptic surfaces with $p_g=q=0$
[Theorem 7.6, Chapter I in \cite{FM}]. Therefore the Baire's category argument applies and we can compare the dimension of moduli of plane curves with the dimension of the Hurwitz' scheme defined by $f.$
 Since very general plane curves are obtained by $\alpha$-covers of $\bP^1$, we must have in fact
\[ N+r-3\geq (d+1)(d+2)/2-9.\]
This implies $r\geq (d+1)(d+2)/2-N-6$. Then combining it with the above equation, we conclude
\[ d(d-3)\geq (d-1)(N/2-2)+(d+1)(d+2)/2-N-6.\]
So $d^2-9d+18\geq (d-3)(N-4).$ Since we have $d\ge 4$ it follows $N\leq d-2.$
\end{proof}

Now we want to estimate the dimension of the moduli space $M(S)$ of irrational minimal elliptic surfaces $S$ with $p_g=q=0$ when the number $N$ of the multiple fibers is fixed. We note that the moduli space $M(S)$ of $S$ is irreducible when we fix the type of the multiple fibers [Theorem 7.6, Chapter I in \cite{FM}]. Since irrational minimal elliptic surfaces $S$ with $p_g=q=0$ can be constructed by logarithmic transforms of rational elliptic surfaces [Section 1.6 in \cite{FM}] and the non-isotriviality of elliptic fibration is preserved by logarithmic transforms, a general element in $M(S)$ has a non-isotrivial elliptic fibration. If an elliptic fibration is not isotrivial then it yields the $j$ invariant of the smooth fibers to be non-constant. Therefore, it is enough to estimate the dimension of the moduli space $M(S)$ of irrational minimal elliptic surfaces $S$ with $p_g=q=0$ under the assumption that the elliptic fibration is not isotrivial.

\begin{lemma} Moreover, we assume that the elliptic fibration is not isotrivial. Let $\bar F$ be any fiber of the elliptic fibration.
Then we have  \label{stima} $h^0(\Omega_S^1(n\bar F))=n-1$  for $n> 0$.
 \end{lemma}
\begin{proof}
We remark that it is enough to prove $h^0(\Omega_S^1(nF))=n-1$ for a general fiber $F$ since the base curve is $\bP^1.$

We take the conormal sequence of the immersion of $F$ into $S:$
\begin{equation}\label{dpi}
0\to N^\ast_F\equiv \cO_F\to{\Omega^1_S}|_F\to \Omega_F^1\to 0.\end{equation}
Since the surface is elliptic, $K_S$ restricts to 0 on $F$ and therefore the sequence (1) is the same as the tangent sequence
\[ 0\to\cO_F\to T_S|_F\to \cO_F\to 0.\]
It follows that the associated coboundary map $\partial:H^0(F,\Omega_F^1)\to H^1(F,\cO_F)$ is the Kodaira-Spencer map.
Since the elliptic fibration is not isotrivial, the coboundary map $\partial$ is not zero.
Therefore the natural map $H^0(\cO_F)\to H^0(F,\Omega^1_S|_F)$ is an isomorphism.
This gives the first step of our proof.

Let $D=F_1+\cdots +F_n$, where $F_i$ are general fibers for $i=1,\ldots, n$. We get then the natural map $$H^0(\cO_D)\to H^0(\Omega^1_S|_D)$$ is
an isomorphism. Next we consider the commutative diagram
\begin{equation*}
\xymatrix{
& 0\ar[r] & \ar[d] \ar[r]  \Omega^1_S &\ar[d] \ar[r] \Omega^1_S(\log D)  \ar[d] \ar[r] &\cO_D \ar[d] \ar[r] &0\\
& 0\ar[r] & \Omega^1_S  \ar[r]  &  \Omega^1_S(D)  \ar[r] & \Omega^1_S|_D \ar[r]   &  0
 }\end{equation*}
One has that the connecting homomorphism $\bar\partial: H^0(F,\cO_F)\to H^1(S, \Omega_S^1)$ in the log sequence $0\to\Omega^1_S\to\Omega^1_S(\log F)\to \cO_F\to 0$ is not trivial since $\bar\partial(1)=c_1(F).$
Therefore the connecting homomorphism $\partial':H^0(\cO_D)=H^0(\cO_{F_1})\oplus\cdots\oplus H^0(\cO_{F_n})\to H^1(\Omega^1_S)$
in the above log diagram maps the generator $1\in H^0(\cO_{F_i})$ to $c_1(F_i)=c_1(F)$, so $\partial'$ has rank 1. Since $H^0(\cO_D)\to H^0(\Omega^1_S|_D)$ is an isomorphism
by the first step of our proof, and since $h^0(\Omega^1_{S})=0$ we
conclude that
$h^0(\Omega^1_{S}(\log D))=h^0(\Omega^1_{S}(D))=h^0(\Omega^1_S(nF))=n-1$.
\end{proof}

\begin{remark}
A similar result was proved for rational elliptic surfaces in   (\cite{LP}, Lemma(2)). Moreover we note that Lemma~\ref{stima} follows under the only assumption that the fibration is not isotrivial.
 \end{remark}
We recall the canonical divisor formula (see \cite{bpv}~[Chapter V, Corollary (12.3)]) for elliptic fibrations:
\begin{equation} \label{formula}
K_S=-F+\sum_i^N(k_i-1)F_i= (N-1)F-\sum_i F_i.\end{equation}

\begin{proposition} \label{moduli-elliptic}
Let $S$ be a minimal elliptic surface with $p_g=q=0$ of Kodaira dimension $1$. Moreover, we assume that the elliptic fibration is not isotrivial.
Let $\cM$ be a Kuranishi space of the deformation space of $S$. Then $M(S)=\dim \cM \leq 8+N.$
\end{proposition}
\begin{proof}

We would like to estimate $h^1(T_S)=\dim H^1(S,T_S)$.
We recall that $H^0(S,T_S)=0.$  In fact the $j$ invariant is not constant.
This implies that the connected component of the identity of the automorphism group of $S$ must commute with the fibration. So if it is not trivial then all smooth fibers are isogeneous, again it contradicts the non-triviality of the $j$ invariant.
Therefore from Hirzebruch-Riemann-Roch Theorem and Serre duality (see \cite{cata}) we have
\[
h^1(T_S)=h^2(T_S)+10 \chi(\cO_S)-2K^2_S=10+h^0(\Omega_S^1(K_S)).
\]
We get
$h^0(\Omega^1_{S}(K_S))=h^0(\Omega^1_{S}((N-1)F-\sum F_i))\leq h^0(\Omega^1_{S}((N-1)F)) \leq N-2$
using Lemma~\ref{stima}.
We get
 $h^2(T_S)\leq N-2$ and then $ M(S)\leq 8+N.$
\end{proof}

\begin{proposition} \label{plane-elliptic}
A very general smooth plane curve of degree $d \geq 8$ cannot be birationally immersed into any elliptic surface $S$ with $p_g=q=0$ of Kodaira dimension $1$.
\end{proposition}
\begin{proof} By Lemma~\ref{elliptic-1} we may assume that $N\leq d-2$, we first observe that the number of moduli of $S$ is by Proposition~\ref{moduli-elliptic}
\[M(S)\leq 8+N\leq 6+d< \frac{(d+1)(d+2)}{2}-9 \,\, \text{for} \,\, d\geq 5.\] It follows that $\kappa(C)$ cannot be rigid in $S.$

By the similar argument in the proof of Proposition~\ref{gen} and using Proposition~\ref{moduli-elliptic}
 \[ \frac {(d+1)(d+2)}{2}-9-8-N < g-\frac{\deg \kappa^\ast(K_S)}{2}.\]
  Therefore by Lemma~\ref{elliptic-1} we have
  \[\frac {(d+1)(d+2)}{2}-\frac {(d-1)(d-2)}{2}< 17+d-2-\frac{\deg \kappa^\ast(K_S)}{2}:\]
 \[ 2d< 15-\frac{\deg \kappa^\ast(K_S)}{2}.\]

 Therefore, since $K_S$ is nef, if $d\geq 8$ then there is no birational immersion $\kappa: C\to S.$
 \end{proof}

Combining Propositions~\ref{gen}, \ref{Enriques}, and \ref{plane-elliptic}, we get Theorem~\ref{immersion}.

\section{Elliptic surfaces without torsion versus plane curves}

This section treats a birational immersion from a very general smooth plane curve to some special families of elliptic surfaces.  We keep the same assumption made at the beginning of Section 2, that is, we  assume that $S$ is an elliptic surface with $p_g=q=0$ and the Kodaira dimension $1,$ but now we will also assume that

$\bullet$ \,\, ${\rm Pic}(S)$ is torsion free. \\

In other words, since $q=0$ we assume the vanishing of the first integral homology group of $S:$ $H_1(S,\bZ)=0.$

Let  $\equiv$ denote the linear equivalence of divisors which, in our case, is the homological equivalence.
Let $F$ be a general fiber of $\pi$ and $\{F_i\}_{\{i=1,\dots,N\}}$ the multiple fibers, i.e.
$F_i$ is effective and $k_i F_i\equiv F$ where $k_i\in \bZ, \ k_i>1.$ We remark from (\cite{Dol}, Chapter 2) that
\[ {\rm Tors}({\rm Pic}(S))={\rm ker}(\oplus_{i=1}^N \bZ/k_i\bZ \stackrel{\psi}{\rightarrow} \bZ/M\bZ)\]
where $M=\prod_{i=1}^{N}k_i$, $\psi((a_1, \ldots, a_N))=\sum a_iM_i$ (mod $M$), $M_i=M/k_i$. It follows $k_i$ and $k_j$ are relatively prime. We  may assume
$1<k_1< k_2<\dots <k_N.$

Since $S$ is not rational, we have $N\geq 2$ by the canonical divisor formula (\ref{formula}).
Let $\Lambda\subset {\rm Pic}(S)$  be the subgroup generated by the $F_i;$ $\Lambda$ is cyclic and let $\lambda$ be the generator such that $F=M\lambda$ with $M >0.$
We have
\[ F\equiv M \lambda,\ \ \ F_i=M_i\lambda. \]
 Therefore if we set $K_S=\rho \lambda$ then
\[\rho=((N-1)M-\sum_{i=1}^N  M_i) =M(N-1-\sum_{i=1}^N \frac{1}{k_i}).\]

Set $k_0=1$ we have $\rho=Ms,$ where
\[s= (N-\sum_{i=0}^N \frac{1}{k_i}).\]

One has immediately
\begin{proposition} \label{ro}
\begin{enumerate}
\item $\rho\geq 3$ with the only one exception $k_1=2,$ $k_2=3.$
\item If $N>2$ then one has $s>1$ with the only exception  $ k_1=2,$ $k_2=3$ $k_3=5.$
In this case $K_S=29\lambda$.
\item If $N=2$ then $K_S=((k_1-1)(k_2-1)-1)\lambda.$
\item If $N>2$ then we have $\rho>6N.$
\end{enumerate}
\end{proposition}
\begin{proof}
Only the last part needs a comment. When $N=3$ with $k_1=2,$ $k_2=3$ $k_3=5$, we have
$\rho=29 > 18.$
Otherwise we have \[\rho\geq \prod_ik_i,\] it follows (it is a very rough estimate) $\rho> N!\geq 6N.$

\end{proof}
\begin{remark}
One has
\begin{enumerate}
\item $K_S=\lambda$ $\iff$ $k_1=2$ and $k_2=3.$
\item If $k_1=2$ and $k_2=k$ then $\rho=k-2.$
\item If $k_1=3$ and $k_2=4$ then $\rho=5$ and $2K_S=F_1+2F_2.$
\end{enumerate}
\end{remark}

\begin{theorem} \label{elliptic}
Let $C$ be a very general smooth plane curve of degree $d\geq 6$. Then there is no birational immersion $\kappa$ from $C$ into any elliptic surface $S$ with $p_g=q=0$ of Kodaira dimension $1$ if ${\rm Pic}(S)$ is torsion free.
\end{theorem}
\begin{proof}
We keep the notation of the previous section. Let $\pi: S\to\bP^1$ be the elliptic fibration, and let $N$ be the number of the multiple fibers.

The proof is given by contradiction. We assume that $\kappa: C\to S$ is a birational immersion where
$C$ is a smooth plane curve of degree $d>4.$

We will apply Proposition~\ref{cliff} and Lemma~\ref{elliptic-1}. We first observe that since $d\geq 6$ the number of moduli of $S$ is
\[M(S)\leq 8+N\leq 6+d< \frac{(d+1)(d+2)}{2}-9.\] It follows that $\kappa(C)$ cannot be rigid in $S.$
Arguing as in Proposition~\ref{gen} we get again using Proposition~\ref{cliff}
 \[\frac {(d+1)(d+2)}{2}-\frac {(d-1)(d-2)}{2}<9+M(S)-\frac{\deg \kappa^\ast(K_S)}{2}.\]
That is
 \[3d<9+8+N -\frac{\deg \kappa^\ast(K_S)}{2}=17+N-\rho\frac{\deg \kappa^\ast(\lambda)}{2}.\]

We have $ \kappa^\ast(K_S)=\rho \lambda.$  Since a positive multiple of $\lambda$ is $F$ moves linearly, so we have
$\deg\kappa^\ast(K_S)\geq \rho.$
Finally we get
\[3d <17+N- \frac{\rho}{2}.\]

When $N= 2$, $\rho=(k_1-1)(k_2-1)-1$ and $3d <19$. So $d\leq 6.$ When $N>2$ then $\rho/2>3N$, and it follows $3d <17-2N$. So $d<4.$

Now we have only to consider the case $d=6$ and $N=2$. This is possible only for $\rho=1$ since $3d<19-\frac{\rho}{2}$, and by Proposition~\ref{ro} this happens precisely when
$N=2, k_1=2, k_2=3.$

We have $\kappa^\ast(K_S)=\cO_C(p)$ for $p\in C.$  Set $\kappa(C)=D$, then we have
$D\cdot K_S=1$ and $D\cdot F=6.$ The composition  $\pi\circ \kappa: C\to \bP^1$ has degree $6$. Since the only rational maps of degree $6$ of a smooth curve $C\subset \bP^2$ are obtained by projecting from a point of $x\in \bP^2$ we obtain
$\cO_C(6p) =\kappa^\ast \cO_S(F)$  is the $g_6^2$ that gives the embedding of $C\to \bP^2.$ But this implies that $p$ is a flex point of the maximal order
$6$ on $C$, which contradicts the fact that $C$ is a very general plane curve of degree $6.$
(Alternatively intersecting $C$ with the multiple fibers we find on $C$  a $3$-tangent and a bitangent to $2$ flexes).
\end{proof}

\begin{remark} One can prove that a very general curve of genus $\geq 7$ (in the sense of moduli)
cannot be birationally immersed in any elliptic surface with $p_g=q=0$ of Kodaira dimension $1$ if ${\rm Pic}(S)$ is torsion free:

We consider the map $f=\pi\circ \kappa: C\to \bP^1.$ Let $\alpha$ be the degree of $f$. Then we have
$\alpha=\kappa(C)\cdot F= M(\kappa(C)\cdot \lambda).$
We remark that the points $p\in C$ such that $\kappa(p)\in F_i$ are ramification points of $f$ with multiplicity
a multiple of $k_i.$
Let $g$ be the genus of $C$. From the Hurwitz formula we have
\[ 2g-2\geq -2\alpha+ \sum^N_{i=1} (\alpha-\alpha/k_i)= \alpha((N-2)-\sum^N_{i=1}{1/k_i}). \]

Suppose $N> 4$. Clearly $\alpha\geq M.$ Since the $k_i$ are pairwise relatively prime,
we get:
\[ 2g-2\geq M(N-2 -(\frac{1}{2}+{\frac{1}{3}+\frac{1}{5}+\frac{1}{7}})-(N-4)\frac{1}{11})> (N-4)\frac{10M}{11} \]

Now since $k_1\geq 2$ and $k_i>2(i-1)$ we get \[M> 2^N (N-1)!\geq 2^N (N-1)^{\frac{N-1}{2}}.\]
Therefore we have
\[ 2g-2> (N-4)\frac{10M}{11}>  \frac{10}{11}2^N(N-4)(N-1)^{\frac{N-1}{2}}.\]
It implies that $N\leq \max {(4, \sqrt g)}$. By Propositions~\ref{gen}, \ref{moduli-elliptic}, and the proof of Theorem~\ref{elliptic}, we have the inequality
$3g-3 \leq g-1+ N+8$. It implies that $2g-2\leq \max({12, 8+ \sqrt g})$. This gives $g\leq 7.$

If $g=7$ then the above argument implies that $N=4$. But this cannot happen because it contradicts the Hurwitz formula
\[ 2g-2=12\geq 2\cdot 3\cdot 5\cdot 7(2-(\frac{1}{2}+\frac{1}{3}+\frac{1}{5}+\frac{1}{7})).\]
\end{remark}
\section{Proof of Main Theorem}
Now we are ready to prove our main Theorem (Theorem~\ref{main} in Introduction). We will give a proof by contradiction. Assume that a dominant rational map $f: X\dashrightarrow Y$ exists.

From Section 3.5 in \cite{gp} we may assume:
\begin{enumerate}
\item  $p_g(Y)=q(Y)=0$ and $Y$ is simply connected.
\item ${\rm Pic}(X)=\bZ {[L]},$ where $L$ is the hyperplane section of $X$.
\end{enumerate}

From the classification theory of algebraic surfaces, we have two cases for $Y:$ either $Y$ is of general type, or $Y$ is an elliptic surface
with Kodaira dimension $1.$

\begin{lemma} Let $C$ be a general hyperplane section of $X$. If $f: X\dashrightarrow Y$ is dominant then $f_C:C\to Y$ is birational onto its image.\label{sla}
\end{lemma}
\begin{proof} Since $C$ is a general hyperplane of $X$,  we may assume that a general point of $Y$ belongs $f_C(C).$
We recall that the Jacobian  of a general hyperplane section $C$ is simple. Then if $f_C: C\to f_C(C)$ is not birational then
the normalization of $f_C(C)$ is rational. Since $Y$ is not ruled it follows that $f$ is not dominant. Therefore it contradicts the assumption.
\end{proof}

\begin{proof} {\bf of Theorem~\ref{main}}.
Assume by contradiction that for a very general surface of degree $d$ we have a dominant rational map
to $Y.$ We get a birational immersion from a very general plane curve of degree $d$ into a surface $Y$ with $p_g=q=0$. The surface $Y$ is
either of general type or an elliptic surface
with Kodaira dimension $1.$ In the first case we have $d\leq 9$ by Proposition~\ref{gen}. But from Theorem 4.2.1 in \cite{gp} this is impossible.
If $Y$ is an elliptic surface then we get $d< 6$ by Theorem~\ref{elliptic}.

So $d=5$ is the only remaining case. After a resolution of the singularities of $f$, we get
\[\xymatrix{
  Z \ar[rr]^{p} \ar[dr]_{g} && X \ar@{.>}[dl]^{f} \\
  & Y.
}\]

Let $E$ be the exceptional divisor. We may assume [Section 3.1 in \cite{gp}]:
\begin{enumerate}
\item $g^\ast K_Y= rL-W$ where $r\geq0$ and $W$ is an effective divisor supported on $E.$
\item $ K_Z=g^\ast K_Y +R,$ $R= sL+E+W,$ $r+s=d-4=1$, and $s\geq 0.$
\end{enumerate}

Since the curve $F$ moves on $Y$, we get $r>0$. And since $g^\ast (K_Y)=\rho g^\ast(\lambda)$ we have
$\rho\leq r\leq d-4=1$. Therefore $\rho=1$ and $s=0.$  We have then
$N=2$, $k_1=2,$ and $k_2=3.$ The branch divisor
is then contained in the image $g(E)$ of the exceptional divisor $E$. This is a rigid divisor, since the rational curves cannot move on the non-ruled surface $Y.$ Let $M(X)$ be the number of moduli of $X$
and let $M(Y)$ be the number of moduli of $Y$. Proposition~\ref{moduli-elliptic} implies
$10\ge M(Y)\geq M(X)=40.$ Then we get a contradiction.
\end{proof}



\begin{thebibliography}{1}

\bibitem{ac}
E.~Arbarello and M.~Cornalba,
\newblock {\em Su una congettura di Petri}.
\newblock {Comment. Math. Helvetici 56 (1981), 1--38.}

\bibitem{ACGH}
E.~Arbarello, M.~Cornalba, P.~Griffiths and J.~Harris,
\newblock {Geometry of algebraic curves. Volume I}.
\newblock {Grundlehren der Mathematischen Wissenschaften 267. Springer-Verlag, 1985.}

\bibitem{ACG}
E.~Arbarello, M.~Cornalba and P.~Griffiths,
\newblock { Geometry of algebraic curves. Volume II. With a contribution by Joseph Daniel Harris}.
\newblock {Grundlehren der Mathematischen Wissenschaften 268. Springer-Verlag, 2011.}

\bibitem{bpv}
W.~Barth, C.~Peters and A.~Van de~Ven,
\newblock {Compact Complex Surfaces}.
\newblock {Ergebnisse der Mathematik und ihrer Grenzgebiete. Springer-Verlag, 1984.}

\bibitem{BP}
F.~Bastianelli and G.~Pirola,
\newblock {\em On dominant rational maps from products of curves to surfaces of general type}.
\newblock {Bull. Lond. Math. Soc. 45 (2013), 1310--1322.}

\bibitem{cata}
F.~Catanese,
\newblock {\em Moduli of algebraic surfaces}.
\newblock {Theory of moduli (Montecatini Terme, 1985), 1--83, Lecture Notes in Math. {\bf1337}, Springer, Berlin, 1988.}

\bibitem{Dol}
I.~Dolgachev,
\newblock {\em Algebraic surfaces with $q=p_g=0$.}
\newblock {Algebraic surfaces, 97--215 (C.I.M.E. Summer Sch. 1977), {\bf76}, Springer, Heidelberg, 2010.}

\bibitem{Dol2}
I.~Dolgachev,
\newblock {\em On Severi's conjecture on simply connected algebraic
surfaces.}
\newblock {Dokl. Akad, Nauk SSSR, 170(1966), 249-252 (Sov. Math. Doklady, 7 (1966), 1169--1172).}

\bibitem{FM}
R.~Friedman and J. Morgan,
\newblock {Smooth four-manifolds and complex surfaces}.
\newblock {Ergebnisse der Mathematik und ihrer Grenzgebiete (3), 27. Springer-Verlag, Berlin, 1994.}

\bibitem{Go}
L.~Godeaux,
\newblock {\em Sur une surface alg\'ebrique de genre zero et de bigenre deux.}
\newblock {Atti Accad. Naz. Lincei, 14 (1931), 479--481.}


\bibitem{gp}
L.~Guerra and G.~Pirola,
\newblock {\em On rational maps from a general surface in $\bP^3$ to surfaces of general type}.
\newblock {Adv. Geom. 8 (2008), 289--307.}

\bibitem{HM}
C.~Hacon and J.~McKernan,
\newblock {\em Boundedness of pluricanonical maps of varieties of general type}.
\newblock {Invent. Math. 166 (2006), 1--25. }




\bibitem{LP}
Y.~Lee and J.~Park,
\newblock {\em A simply connected surface of general type with $p_g=0$ and $K^2 =2$}.
\newblock {Invent. Math. 170 (2007), 483--505}.

\bibitem{LeeP}
Y.~Lee and G.~Pirola,
\newblock {\em On rational maps from the product of two general curves}.
\newblock {math.AG.arXiv:1411.4263}.


\bibitem{Ma}
K.~Maehara,
\newblock {\em A finiteness property of varieties of general type}.
\newblock {Math. Ann. 262 (1983), 101--123}.

\bibitem{Ta}
S.~Takayama,
\newblock {\em Pluricanonical systems on algebraic varieties of general type}.
\newblock {Invent. Math. 165 (2006), 551--587}.

\bibitem{Ts}
H.~Tsuji,
\newblock {\em Pluricanonical systems of projective varieties of general type. II}.
\newblock {Osaka J. Math.  44  (2007), 723--764}.






\end{thebibliography}
\end{document}